\documentclass[]{interact}
\usepackage{lmodern}
\usepackage{xcolor}
\usepackage{epstopdf}
\usepackage[caption=false]{subfig}
\usepackage{mathtools} 
\usepackage[numbers,sort&compress]{natbib}
\bibpunct[, ]{[}{]}{,}{n}{,}{,}
\renewcommand\bibfont{\fontsize{10}{12}\selectfont}
\makeatletter
\def\NAT@def@citea{\def\@citea{\NAT@separator}}
\makeatother

\theoremstyle{plain}
\newtheorem{theorem}{Theorem}[section]
\newtheorem{lemma}[theorem]{Lemma}

\newtheorem{proposition}[theorem]{Proposition}
\newtheorem{assumption}[theorem]{Assumption}
\theoremstyle{definition}
\newtheorem{definition}[theorem]{Definition}
\newtheorem{example}[theorem]{Example}

\theoremstyle{remark}

\newcommand{\R}{\mathbb{R}}

\usepackage{algorithm}
\usepackage{algpseudocode,algorithmicx}
\makeatletter
\renewcommand{\ALG@name}{Algorithm}
\makeatother
\usepackage{algcompatible}
\usepackage{physics}
\usepackage{hyperref}

\begin{document}


\title{A Steepest Gradient Method with Nonmonotone Adaptive Step-sizes for the Nonconvex Minimax and Multi-Objective Optimization Problems}

\author{
\name{Nguyen Duc Anh\textsuperscript{a} and Tran Ngoc Thang\textsuperscript{a,}$^\ast$\thanks{$^\ast$Corresponding author. Email: thang.tranngoc@hust.edu.vn}}
\affil{\textsuperscript{a}Faculty of Mathematics and Informatics, Hanoi University of Science and Technology, Hanoi, Vietnam}
}

\maketitle

\begin{abstract}
This paper proposes a new steepest gradient descent method for solving nonconvex finite minimax problems using non-monotone adaptive step sizes and providing proof of convergence results in cases of the nonconvex, quasiconvex, and pseudoconvex differentiate component functions. The proposed method is applied using a referenced-based approach to solve the nonconvex multiobjective programming problems. The convergence to weakly efficient or Pareto stationary solutions is proved for pseudoconvex or quasiconvex multiobjective optimization problems, respectively. A variety of numerical experiments are provided for each scenario to verify the correctness of the theoretical results corresponding to the algorithms proposed for the minimax and multiobjective optimization problems.
\end{abstract}


\section{Introduction}
Nonconvex finite minimax problems, denoted by (MP), involve minimizing a function that is the maximum of a finite number of smooth nonconvex real-valued functions on the space $\mathbb{R}^n.$ The complexity of these problems arises from the non-differentiability and nonconvex nature of the objective function, making them challenging to solve using traditional optimization methods for smooth convex functions. Although there are many challenges, this problem attracts research interest because it has numerous applications in economics \cite{wiecek2007advances}, finances \cite{el2020finance, fliege2014robust}, engineering \cite{eschenauer2012multicriteria},  technology \cite{jimenez2017pareto}, and more.  Furthermore, this problem is closely related to nonconvex multiobjective optimization problems. Finding an optimal solution for a multiobjective optimization problem can be transformed into a finite minimax problem using the Tchebycheff scalarization method. In this paper, we propose a new algorithm to solve (MP), which is then applied to address nonconvex multiobjective optimization problems using a reference-based approach. Many algorithms have been proposed to solve (MP) such as the steepest gradient descent method \cite{pshenichny1978numerical}, sequential quadratic programming  \cite{charalambous1978efficient, jian2020globally}, interior point methods \cite{lukvsan2005primal, lukvsan2004interior, vanderbei1999interior}, conjugate gradient algorithms \cite{konnov2020non} and smoothing methods \cite{pee2011solving}.

Recently, the smoothing method has been widely studied and used for solving (MP) (see \cite{pee2011solving}). However, this method does not guarantee convergence for cases with nonconvex component functions, particularly quasiconvex and pseudoconvex functions. Since the algorithms for solving optimization problems in tool libraries used for training deep learning models employ the steepest gradient method, we are interested in algorithms that follow this approach, such as \cite{polak2008algorithm, polak2008convergence, pshenichny1978numerical}. Pshenichny et al. \cite{pshenichny1978numerical} proposed an algorithm based on linearization with regard to the combination of line search and steepest descent technique. The work addresses the minimax problem with convex component functions and proves convergence to a globally optimal solution. However, in the nonconvex case, such as with pseudoconvex component functions, convergence to a stationary point is not guaranteed. We will state and prove the convergence of this algorithm in nonconvex, quasiconvex, and pseudoconvex cases.

When considering the application of the algorithm \cite{pshenichny1978numerical} to machine learning problems, we find that the line search procedure in this algorithm can be time-consuming at each iteration when used in the training process of a model, where the component functions are loss functions. Therefore, we developed a line search-free variant of this algorithm using a nonmonotonic adaptive step size. With such a step size, each iteration only requires evaluating the expression once, and depending on the value of the evaluation, the step size is adjusted—either increased or decreased—at each iteration. The convergence of the variant algorithm in nonconvex, quasiconvex, and pseudoconvex cases also will be proven here.

The descent method is another major approach to solve the multicriteria optimization problems. It does not need any information about the parameters beforehand and recent literature has focused a lot on these terms, for example, the methods of steepest descent \cite{fliege2000steepest}, conjugate gradient \cite{lucambio2018nonlinear}, Newton \cite{wang2019extended} and projected gradient \cite{cruz2011convergence, drummond2004projected}. These methods are iterative and use a consistent gradient-based search to guarantee that the objective function values decrease with each iteration. Arbaki \cite{akbari2024new} proposed a nonsmooth non-monotone line search approach to handle the Lipschitz function, corresponding to the global convergence based on some assumptions. Another approach for handling the multicriteria optimization is based on the Tchebycheff formulation. This approach considers solving the minimax problem, i.e., minimization of the maximum of component functions with the direction of concessions \cite{kaliszewski2012interactive, thang2020monotonic}. This method is distinctive due to its reference-based characteristics, meaning that starting from an initial vector, it enables the identification of weakly efficient points across the Pareto set, essentially solving the minimax problem. Consequently, the efficient set can be approximated by solving a series of minimax problems with chosen directional adjustments. We introduce an algorithm tailored for addressing multiobjective optimization using the minimax approach and demonstrate its convergence in non-convex, quasiconvex, and pseudoconvex scenarios.


The organization of this paper is outlined as follows: In Section 2, we cover some essential concepts. Next, we propose the algorithms and their characteristics in Section 3 corresponding to their proof of convergence. Formulation between the multiobjective optimization and the minimax problem are provided within Section 4. Section 5 provides the algorithm's numerical experiments for the proposed algorithms. Finally, we conclude our findings and propose potential avenues for future research in the last Section.

\section{Preliminaries}
\subsection{Notations}
Consider the non-empty set $D \subseteq \mathbb{R}^n$, is called convex if
for all $\theta^1, \theta^2 \in D$ và $0 \leq \lambda \leq 1$, we have $\lambda \theta^1+(1-\lambda) \theta^2 \in$ $D$. Given the set $L\subseteq \mathbb{R}^n$, we denote Conv $L$ is the smallest convex shape enclosing the set of points in $L$, i.e. 
\begin{equation*}
    \text{Conv } L = \{\theta \in L: \theta=\sum_{i=1}^k \lambda_i \theta^i, \theta^i \in L, i = \overline{1, k}\},
\end{equation*}
where the real number $\lambda_1, \cdots, \lambda_k \geq 0$ and $\displaystyle\sum_{i=1}^k \lambda_i=1$. 

Consider a function  $g: D \to \overline{\mathbb{R}} = \mathbb{R} \cup \{-\infty , +\infty\},$ defined on the Banach space $D.$  The generalized subdifferential  \cite{penot1997generalized} of function $g$ at $\theta\in D$ is described in the dual space $D^*$  as follows:
\[
\upartial^{\uparrow} g(\theta) = \left\{ \theta^* \in D^* : \left\langle \theta^*, v \right\rangle \leq g^{\uparrow}(\theta, v), \forall v \in D \right\},
\]
where the generalized Clarke-Rockafellar derivative of $g$ at $\theta$ along the direction $v$ is  
\[
g^{\uparrow}(\theta, v) = \sup_{\varepsilon > 0} \lim_{(\xi, \alpha) \downarrow_g \theta; t \to 0^+} \sup_{u \in B(v, \varepsilon)} [g(\xi + t u) - \alpha] / t,
\]
in there $(\xi, \alpha) \downarrow_g \theta$ represents $\xi \rightarrow \theta$, $\alpha \rightarrow g(\theta)$, $\alpha \geq g(\xi)$.

Given the non-empty set $D \subseteq \R^n$, a continuous function $g$ is said to be
\begin{itemize}
    \item[a)] \textit{convex} on $D$ if, given any $\theta, \xi \in D$, then
    \begin{equation*}
        g\left(\lambda \theta+(1-\lambda) \xi\right) \leq \lambda g\left(\theta\right)+(1-\lambda) g\left(\xi\right),
    \end{equation*}
    for all $\lambda \in [0,1];$
    \item[b)] \textit{pseudoconvex} on $D$ if, given any $\theta, \xi \in D$, then
\[
    \forall \theta, \xi \in X: g(\xi) < g(\theta) \rightarrow \forall \theta^* \in \upartial^{\uparrow} g(\theta): \langle \theta^*, \xi - \theta \rangle < 0;
\]

    \item[c)] \textit{quasiconvex} on $D$ if, given any $\theta, \xi \in D$, then
    \begin{equation*}
    g(\lambda \theta + (1-\lambda) \xi) \le \max\{ g(\theta) , g(\xi) \},
    \end{equation*}
    for all $\lambda \in [0,1].$
\end{itemize}

In the case that $g$ is locally Lipschitz, the generalized Clarke-Rockafellar subdifferential is similarly defined as follows \cite{mifflin1977semismooth}: 
\[
\upartial^C g(\theta) = \left\{ \theta^* \in D^* : \left\langle \theta^*, v \right\rangle \leq g^0(\theta, v), \forall v \in D \right\},
\]
with  
\[
g^0(\theta, v) = \limsup_{\zeta \rightarrow \theta, t \to 0^+} [g(\zeta + t v) - g(\zeta)] / t.
\]
A function $g$ is called \textit{quasidifferentiable} \cite{mifflin1977semismooth} at $\theta$ if $g'(\theta, d)$ exists and equals $g^0(\theta, d)$ for each $d \in \R^n$, i.e.
\begin{equation*}
    g'(\theta, d) = g^0(\theta, d), \quad \forall d \in \R^n,
\end{equation*}
where we denote $g'(\theta, d)$ by the directional derivative of $g$ at $\theta$ along the direction $d$ as follows
\begin{equation*}
    g'(\theta, d) = \lim_{t \to 0^+} \dfrac{g(\theta^0 + td) - g(\theta^0)}{t}.
\end{equation*}

Consider a function $g: \R^n \to \R,$ locally Lipschitz at the open ball center $\theta$  and quasidifferentiable at $\theta$. The function $g$ is called \textit{semiconvex} \cite{mifflin1977semismooth} at $\theta \in X \subseteq \R^n$ if 
 $\theta + d \in X$ and $g'(\theta, d) \ge 0$ leads to $g(\theta+d) \ge g(\theta)$.

\subsection{Basic Results}
Consider the continuous function $g: \R^n \to \R$ and sequence of differential functions $g_i: \R^n \to \R$ with respect to its gradient $\nabla g_i(\theta)$ for $i = \overline{1, m}$. The following assertions are used later.
\begin{proposition}{\cite{agrawal2020disciplined}}
\label{boyy}
    Let $g_i : \R^n \to \R$, $i = \overline{1, m}$ be quasiconvex functions. Then, $G(\theta) = \displaystyle\max_i\{g_i(\theta), i = \overline{1, m}\}$ is quasiconvex.
\end{proposition}

\begin{lemma}{\cite{xu2002iterative}}
\label{sequence_lem}
    Let $\left\{u_k\right\} ;\left\{v_k\right\} \subset(0 ; +\infty)$ satisfies
$$
u_{k+1} \leq u_k+v_k \; \forall k \geq 0 ; \quad \sum_{k=0}^{+\infty} v_k<+\infty .
$$
Therefore, a limit exists $\lim _{k \rightarrow +\infty} u_k=c \in \mathbb{R}$.
\end{lemma}

\begin{proposition}{\cite{penot1997generalized}}
\label{prop41}
    If $g$ is quasiconvex and locally Lipschitz, then it is $\partial^{C}$-quasiconvex in the sense that the following condition holds:
$$
\forall \theta, \xi \in X, g(\xi)<g(\theta) \Rightarrow \forall \theta^* \in \partial^{C} g(\theta):\left\langle \theta^*, \xi-\theta\right\rangle \leq 0 .
$$
\end{proposition}

\begin{proposition}{\cite{mifflin1977semismooth}} If $g$ is pseudoconvex and differentiable, then $g$ is a semiconvex function.
    \label{pseudo_semi}
\end{proposition}

\begin{proposition}{\cite{mifflin1977semismooth}}
    Let $g: \R^n \to \R$ be differentiable continuous, then $g$ is locally Lipschitz, $\partial^C g(\theta) = \{ \nabla g(\theta)\}$ for each $\theta \in \R^n$, and $g$ is quasidifferentiable over $\R^n$.  
    \label{smooth_clarke}
\end{proposition}

\begin{theorem}{\cite{mifflin1977semismooth}}
\label{semismooth}
    Assume that $g_1, g_2, \ldots, g_m$ are locally Lipschitz functions. $X$ is a subset of $\mathbb{R}^n$. Let $G(\theta) = \displaystyle\max_i\{g_i(\theta), i = \overline{1, m}\}$ and $J(\theta) = \{ i : g_i(\theta) = G(\theta) \}$.
    \begin{itemize}
        \item[i)] $g$ is Lipschitz continuous. Moreover, for each $\theta \in \R^n$,
        \begin{equation*}
            \partial^C G(\theta) \subset \text{ Conv}\{\partial^C g_i(\theta), i \in J(\theta) \}.
        \end{equation*}
        \item[ii)] If $g_1, g_2, \ldots, g_m$ are semiconvex (quasidifferentiable) over $X$, then $G$ is semiconvex (quasidifferentiable) over $X$ and for each $\theta \in \mathbb{R}^n$, 
        \begin{equation*}
            \partial^C G(\theta) = \text{ Conv}\{ \partial^C g_i(\theta): i \in J(\theta) \}.
        \end{equation*}
    \end{itemize}
\end{theorem}

In the next Section, we present our algorithm designed for solving the \eqref{MP} problem. This algorithm is developed using non-monotone adaptive strategies combined with the steepest descent method. 

\section{Steepest Gradient Algorithms with Nonmonotone Adaptive Step-sizes for Problem (MP)}
\subsection{Minimax problem}
Consider the functions $g_i: \R^n \to \R$, $i = \overline{1, m}$ be continuously differentiable. Let $G(\theta) = \displaystyle\max_i\{g_i(\theta), i = \overline{1, m}\}.$ The nonconvex finite minimax problem is considered as follows:
\begin{alignat*}{2}\label{MP}
\min & \; G(\theta), \tag{MP} \\
\text {s.t } & \theta \in \R^n. 
\end{alignat*}
\begin{assumption}
We consider the following assumptions:
\begin{itemize}
    \item[(A1)] The Problem \eqref{MP} has the solution.
    \item[(A2)] The domain $\Omega=\left\{\theta: G(\theta) \leqslant G\left(\theta^0\right)\right\}$ be bounded.
    \item[(A3)] $\nabla g_i(\theta), i = \overline{1, m}$ satisfying the Lipschitz condition in $\Omega$ with constant $L$.
\end{itemize} 
\end{assumption}

\subsection{The nonmonotone adaptive steepest gradient algorithm}
Firstly, we provide our algorithm for solving the \eqref{MP} problem, based on non-monotone adaptive strategies and steepest descent method, as being described below:

\begin{algorithm}[H]
\caption{Adaptive Non-monotone strategies for solving the \eqref{MP} problem}
\hspace*{\algorithmicindent} \textbf{Input:} Approximation point $\theta^0$\\
\begin{algorithmic}[1]
\State Initialize $\varepsilon \in \left(0, \dfrac{1}{2}\right), \alpha_0 = 1, \delta > 0, \sigma \in (0,1), \{\eta_k\} \in (0,1)$, such that $\displaystyle \sum_{k=0}^{+\infty} \eta_k < +\infty$, iteration $k=0$, $s= 0$.

\State At the iteration $k$, $p_k$ is the solution of the subsidiary problem \eqref{sub_prob} at $\theta^k$, that is

\begin{equation}
    \beta + \dfrac{1}{2}\norm{p}^2, \label{sub_prob}\tag{SP}
\end{equation}
with the constraints as follows
\begin{equation*}
    \langle \nabla g_i(\theta), p \rangle + g_i(\theta) - \beta(\theta) \le 0, i \in J_{\delta}(\theta),
\end{equation*}
in there, $\delta>0$ và $J_{\delta}(\theta) \coloneqq \{j = \overline{1, m}: g_j(\theta) \ge G(\theta) - \delta\}$.

\IF{$G(\theta^k + \alpha_k p_k) \le G(\theta^{k}) - \alpha_k \varepsilon \norm{p_k}^2$}
    \STATE $\alpha_{k+1} = \alpha_k + \eta_k \sigma^{s}$
\ELSE
    \STATE $\alpha_{k+1} = \alpha_k \sigma$.
    \STATE $s = s + 1$
\ENDIF

\State Update $k = k+1$ and $\theta^{k+1} = \theta^k + \alpha_{k} p_k$. If $\theta^{k+1} = \theta^k$ then \textit{STOP} else proceed to \textbf{\textit{Step 2}}.
\end{algorithmic}    

\hspace*{\algorithmicindent} \textbf{Output:} Optimal Solution of the \eqref{MP} problem.
\label{algo1}
\end{algorithm}

Following this, we will present the convergence analysis of Algorithm \ref{algo1}, addressing the three scenarios: non-convex, quasiconvex, and pseudoconvex, under certain mild assumptions.

\subsection{Convergence of the algorithms}








\begin{lemma}\label{lemma33-1}
    The sequence of step-size $\{\alpha_k\}$ generated by the Algorithm \ref{algo1} is not converge to $0$. 
\end{lemma}
\begin{proof}
Denote $p_k \equiv p(\theta^k), \beta_k \equiv \beta(\theta^k)$, in there $\theta^k$ is the approximation at iteration $k$.

Let $i \in J_{\delta}(\theta^k)$, consider the generalized langrange expansion as follows
\begin{align*}
    g_i(\theta^k+\alpha p_k) &= g_i(\theta^k) + \alpha \langle \nabla g_i(\theta^k), p_k\rangle + \alpha \langle \nabla g_i(\xi_k) - \nabla g_i(\theta^k), p_k \rangle \\
    &\le g_i(\theta^k) + \alpha \langle \nabla g_i(\theta^k), p_k \rangle + L\alpha^2 \norm{p_k}^2,
\end{align*}
in there $\xi_k = \theta^k + \alpha \zeta_1 p_k$, $\zeta_1 \in (0,1)$. 

For $i \notin J_{\delta}(\theta^k)$, that is,
\begin{align*}
    g_i(\theta^k + \alpha p_k) &= g_i(\theta^k) + \alpha \langle \nabla g_i(\xi_{k0}), p_k \rangle \\
    &\le g_i(\theta^k) + \alpha K \norm{p_k} \\
    &\le G(\theta^k) - \delta + \alpha K \norm{p_k},
\end{align*}
in there $K = \max_{\theta \in \Omega} \norm{\nabla g_i(\theta)}$ and $\xi_{k0} = \theta^k + \alpha\zeta_2p_k, \zeta_2 \in (0,1)$.

Consider $i \in J_{\delta}(\theta^k)$, we have
\begin{align*}
    g_i(\theta^k + \alpha p_k) &\le g_i(\theta^k) + \alpha \langle \nabla g_i(\theta^k), p_k \rangle + L\alpha^2 \norm{p_k}^2 \\
    &\le g_i(\theta^k) + \alpha (\beta(\theta^k) - g_i(\theta^k)) + L\alpha^2 \norm{p_k}^2 \\
    &\le (1 - \alpha) g_i(\theta^k) + \alpha \left(G(\theta^k) - \dfrac{1}{2} \norm{p_k}^2\right) + L\alpha^2 \norm{p_k}^2 \\
    &\le G(\theta^k) - \dfrac{\alpha}{2} \norm{p_k}^2 + L\alpha^2 \norm{p_k}^2.
\end{align*}

If we choose $0 \le \alpha \le \alpha_k^1$, $\alpha_k^1 = \dfrac{\delta}{\norm{p_k}\left( K + \dfrac{1}{2} \norm{p_k}\right)}$, then we have
\[
    g_i(\theta^k + \alpha p_k) \le G(\theta^k) - \dfrac{\alpha}{2} \norm{p_k}^2 + L\alpha^2 \norm{p_k}^2 \quad i = \overline{1, m}, 
\]
which is equivalent to 
\[
    G(\theta^k + \alpha p_k) \le G(\theta^k) - \dfrac{\alpha}{2} \norm{p_k}^2 + L\alpha^2 \norm{p_k}^2, \quad i = \overline{1, m}.
\]

For $i = \overline{1, m}$,
\[
    G(\theta^k + \alpha p_k) \le G(\theta^k) - \alpha \norm{p_k}^2\left(\dfrac{1}{2} - \alpha L \right),
\]
and let $0 \le \alpha \le \alpha_k' \coloneqq \dfrac{\tfrac{1}{2} - \varepsilon}{L}$, we obtain the following statement
\[
    G(\theta^k + \alpha p_k) \le G(\theta^k) - \alpha \norm{p_k}^2\varepsilon; \quad i = \overline{1, m}.
\]

In summary, with $0 \le \alpha \le \overline{\alpha}_k$, $\overline{\alpha}_k = \min\{1, \alpha_k^1, \alpha_k'\}$, we have 
\[
    G(\theta^k + \alpha p_k) \le G(\theta^k) - \alpha \norm{p_k}^2\varepsilon.
\]

Our assumption is that there exists a solution of the (MP) problem, i.e. $\norm{p_k} \not\to +\infty$ when $k \to +\infty$. Next, we prove that $\overline{\alpha}_k$ is not converge to $0$, which means for each $k$, there exists $\overline{\alpha}'$ such that for every $\alpha \in (0, \overline{\alpha}')$, we have 
\[
    G(\theta^k + \alpha p_k) \le G(\theta^k) - \alpha \norm{p_k}^2\varepsilon.
\]

Let $\overline{\alpha} = \inf_k\{\overline{\alpha}_k\}$, then for all $\alpha \in (0, \overline{\alpha})$, we have
\begin{equation*}\label{eq2}\tag{4}
    G(\theta^k + \alpha p_k) \le G(\theta^k) - \alpha \norm{p_k}^2\varepsilon, \quad \forall \;k.
\end{equation*}

Now, we prove that $\overline{\alpha} \neq 0$. On the contrary, assume $\overline{\alpha} = 0$. Based on the definition of infimum, we have $
\overline{\alpha} \le \overline{\alpha}_k, \forall k$, and $\forall \varepsilon_1 > 0, \exists k_0 \in \mathbb{N}, 
\overline{\alpha} + \varepsilon_1 > \overline{\alpha}_{k_0}
$. 

We let $\epsilon_1$ tends to $0$. Hence, there exists a subsequence $\{ \overline{\alpha}_{k_i}\}$ such that $\lim_{i \to +\infty} \overline{\alpha}_{k_i} = 0$, which mean $\lim_{i \to +\infty} \dfrac{\delta}{\norm{p_{k_i}}\left( K + \dfrac{1}{2} \norm{p_{k_i}}\right)} =0$, implies to $\norm{p_{k_i}} \to +\infty$, which is contradiction. Therefore, we have $\overline{\alpha} \neq 0$ and $\alpha_{k}$ is not converge to $0$, which means $\norm{p_k}^2 \to 0$ when $k \to +\infty$. Hence, $p_k \to \textbf{0}$ when $k \to +\infty$. \end{proof}

\begin{lemma}\label{lemma33-2}
At the $k$-th iteration, let us denote
$$\mathbb{I}(G(\theta^k + \alpha_k p_k) - G(\theta^k) + \alpha_k \varepsilon \norm{p_k}^2 > 0) = \begin{cases}
    0, & \text{ if } G(\theta^k + \alpha_k p_k) - G(\theta^k) + \alpha_k \varepsilon \norm{p_k}^2 \le 0; \\
    1, & \text{ if } G(\theta^k + \alpha_k p_k) - G(\theta^k) + \alpha_k \varepsilon \norm{p_k}^2 > 0
\end{cases}$$

Then we obtain $\displaystyle\sum_{k=0}^{+\infty} \mathbb{I}(G(\theta^k + \alpha_k p_k) - G(\theta^k) + \alpha_k \varepsilon \norm{p_k}^2 > 0) < +\infty$.
\end{lemma}
\begin{proof}
We will prove this lemma by contradiction. Suppose $\displaystyle\sum_{k=0}^{+\infty} \mathbb{I}(G(\theta^k + \alpha_k p_k) - G(\theta^k) + \alpha_k \varepsilon \norm{p_k}^2 > 0) = +\infty$. Next, we will demonstrate that at iteration $k$, we deduce
\[
    \alpha_{k} \le \sigma^{\overline{k}}\left(\alpha_0 + \displaystyle\sum_{i = 0}^{k-1} \eta_i\right),
\]
with $\overline{k} \coloneqq |\{j \in \mathbb{N}, j \le k-1: \mathbb{I}(G(\theta^j + \alpha_j p_j) - G(\theta^j) + \alpha_j \varepsilon \norm{p_j}^2 > 0) = 1\}|$. We will prove this statement by induction. The statement satisfies when $k = 0,1$. Let assume the correctness until $k = l$, i.e, $\alpha_{l} \le \sigma^{\overline{l}}\left(\alpha_0 + \displaystyle\sum_{i = 0}^{l-1} \eta_i\right)$. Now we will prove in case $k = l+1$. Let's consider the two following cases:
\begin{enumerate}
    \item $\mathbb{I}(G(\theta^l + \alpha_l p_l) - G(\theta^l) + \alpha_l \varepsilon \norm{p_l}^2 > 0) = 0$. In this case, we obtain 
     \begin{align*}
\alpha_{l+1} &= \alpha_l + \eta_l \sigma^{\overline{l}} \\
&\le \sigma^{\overline{l}} \left( \alpha_0 + \sum_{i=0}^{l-1} \eta_i \right) +  \eta_l \sigma^{\overline{l}} \\
&\le \sigma^{\overline{l+1}} \left( \alpha_0 + \sum_{i=0}^{l} \eta_i \right). 
\end{align*}
    \item $\mathbb{I}(G(\theta^l + \alpha_l p_l) - G(\theta^l) + \alpha_l \varepsilon \norm{p_l}^2 > 0) = 1$. Similarly, we deduce
    \begin{align*}
        \alpha_{l+1} &= \sigma \alpha_l \\
                    &\le \sigma \sigma^{\overline{l}} \left( \alpha_0 + \sum_{i=0}^{l-1} \eta_i \right)  \\
                    &\le \sigma^{\overline{l+1}}  \left( \alpha_0 + \sum_{i=0}^{l} \eta_i \right).
    \end{align*}    

\end{enumerate}

Hence, we have $\alpha_{k} \le \sigma^{\overline{k}}\left(\alpha_0 + \displaystyle\sum_{i = 0}^{k-1} \eta_i\right)$ for all $k$. By the assumption $\displaystyle\sum_{k=0}^{+\infty} \mathbb{I}(G(\theta^k + \alpha_k p_k) - G(\theta^k) + \alpha_k \varepsilon \norm{p_k}^2 > 0) = +\infty$, we imply $\sigma^{\overline{k}} \to 0$ and the sum $\alpha_0 
 + \displaystyle\sum_{i = 0}^{k-1} \eta_i$ is bounded, which results in $\alpha_{k} \to 0$ when $k \to +\infty$, in contrast to $\alpha_k \not\to 0$, hence we have $\displaystyle\sum_{k=0}^{+\infty} \mathbb{I}(G(\theta^k + \alpha_k p_k) - G(\theta^k) + \alpha_k \varepsilon \norm{p_k}^2 > 0) < +\infty$ and the procedure ensures that $p_k \to 0, \alpha_k \not\to 0$ when $k \to +\infty$. \end{proof}

\begin{lemma}\label{33-3}
    Any limit point $\theta^*$ of the sequence $\{\theta_k\}$ generated by the Algorithm \ref{algo1}, satisfies the necessary conditions for a minimum of $G(\theta)$ with $\theta \in \R^n$. 
\end{lemma}
\begin{proof}
Consider the subsidiary problem \eqref{sub_prob}, let $i \in J_{\delta}(\theta^k)$, so if $i \notin J_{\delta}(\theta^k)$, we assign $u_k^i = 0$, then the necessary and sufficient condition for optimal solution of \eqref{sub_prob} is displayed as follows:

\[\begin{cases}
    \displaystyle\sum_{i =1}^m u_k^i = 1, \\
    p(\theta^k) + \displaystyle\sum_{i =1}^m u_k^i \nabla g_i(\theta^k) = 0,\\
    u^i_k(\langle \nabla g_i(\theta^k), p_k \rangle + g_i(\theta^k) - \beta(\theta^k)) = 0.
\end{cases}\]

For $i \in J_{\delta}(\theta^k)$ and $g_i(\theta^k) = G(\theta^k)$, we have
\[
\beta_k \ge g_i(\theta^k) - \langle \nabla g_i(\theta^k), p(\theta^k) \rangle \ge G(\theta^k) - K \norm{p_k},
\]
and $\beta_k \le G(\theta^k) - \dfrac{1}{2} \norm{p_k}^2$, so for each accumulation point $\theta^*$ generated from sequence $\{\theta^k\}$, we have $\beta_k \to G(\theta^*)$ when $k \to +\infty$. Then, the accumulation point $\theta^*$ satisfies some following conditions:
\[\begin{cases}
    \displaystyle\sum_{i =1}^m \overline{u}^i = 1, \overline{u}^i \ge 0; \label{Theorem_1} \\[0.7cm]
    \displaystyle\sum_{i =1}^m \overline{u}^i \nabla g_i(\theta^*) = 0,\\[0.7cm]
    \overline{u}^i(g_i(\theta^*) - G(\theta^*)) = 0, i = \overline{1, m}.
\end{cases}\] \end{proof}

\begin{theorem} \label{lemma33-4}
In the non-convex case, i.e. $g_i(\theta), i = \overline{1, m}$ are non-convex, then the accumulation point (if it exists) $\theta^*$ is a stationary point.
\end{theorem}
\begin{proof}
Assume $\theta^*$ exists be accumulation point. Based on the algorithm, $\theta^{k+1} = \theta^k + \alpha_k p_k$, where $\alpha_k$ is bounded and $p_k \to 0$ when $k \to +\infty$. We consider $\forall \varepsilon > 0,$ there exist $N > 0$ such that $\forall m, k  \ge N, m > k$, we have $\norm{\theta^m - \theta^k} < \varepsilon$. Let $m \to +\infty$, then we have
\begin{equation*}
    \norm{\theta^* - \theta^k} < \varepsilon.
\end{equation*}

Then we have
\begin{equation*}
    0 \le \norm{\theta^{k+1} - \theta^k} \le \norm{\theta^{k+1} - \theta^*} + \norm{\theta^* - \theta^k} < 2\varepsilon.
\end{equation*}

Let $\varepsilon \to 0^+$ leads to $k \to +\infty$ that we have $\theta^*$ is stationary point. \end{proof}

\begin{theorem}\label{lemma33-5}
In the quasi-convex case, i.e. $g_i(\theta), i = \overline{1, m}$ are quasiconvex, then the accumulation point exists and is a stationary point.
\end{theorem}
\begin{proof} We observe that 
\begin{align*}
\alpha_k \varepsilon \norm{p(\theta_k)}^2 &= \varepsilon \langle p(\theta_k), \alpha_k p(\theta_k) \rangle \\
&= \varepsilon \langle -p(\theta^k), \theta_k - \theta_{k+1} \rangle \\
&= \varepsilon \left\langle \sum_{i \in J_{\delta}(\theta^k)} u^i \nabla g_i(\theta^k), \theta^k - \theta^{k+1} \right\rangle \\
&= \varepsilon \sum_{i \in J_{\delta}(\theta^k)} u^i \left\langle \nabla g_i(\theta^k), \theta^k - \theta^{k+1} \right\rangle 
\end{align*}

Noting that \(\displaystyle\sum_{i \in J_{\delta}(\theta^k)} u^i \left\langle \nabla g_i(\theta^k), \theta^k - \theta^{k+1} \right\rangle \geq 0\). By using Lemma \ref{lemma33-2} and based on the  assumption about the existence of the solution of the Problem (MP), we infer that
\begin{equation*}
    \sum_{k=0}^{+\infty} \sum_{i \in J_{\delta}(\theta^k)} u^i \left\langle \nabla g_i(\theta^k), \theta^k - \theta^{k+1} \right\rangle < +\infty.
\end{equation*}

For all \(w \in \R^n\), we have
$$
\begin{aligned}
\left\|\theta^{k+1} - w\right\|^2 &= \left\|\theta^k - w\right\|^2 - \left\|\theta^{k+1} - \theta^k\right\|^2 + 2 \left\langle \theta^{k+1} - \theta^k, \theta^{k+1} - w \right\rangle \\
&\leq \left\|\theta^k - w\right\|^2 - \left\|\theta^{k+1} - \theta^k\right\|^2 + 2 \alpha_k \sum_{i \in J_{\delta}(\theta_k)} u^i \left\langle \nabla g_i(\theta^k), w - \theta^{k+1} \right\rangle.
\end{aligned}
$$

Now, suppose that \(g_i\), \(i = \overline{1, m}\) is quasiconvex on \(\Omega\). Using Proposition \ref{boyy}, we deduce \(G\) is quasiconvex. Denote
\[
U := \left\{\theta \in \Omega : G(\theta) \leq G\left(\theta^k \right) \forall k \ge 0\right\},
\]
and
\[
\Omega := \left\{\theta \in \R^n : G(\theta) \leq G\left(\theta^0 \right)\right\}.
\]

Remark that the solution of \eqref{MP} lies in \(U \subseteq \Omega\), so the cardinality of \(U\) is assumed to be non-empty. We have assumed that \(\Omega\) is bounded, then \(U\) is bounded. Let \(\{\theta^l\} \in U\) be a convergent sequence to \(\overline{\theta}\), then \(G(\theta^l) \le G(\theta^k) \;\forall k \ge 0\), so we have \(\lim_{l \to +\infty} G(\theta^l) = G(\lim_{l \to +\infty} \theta^l) = G(\overline{\theta}) \le G(\theta^k), \forall k \ge 0\), because \(G\) is a continuous function. Thus, \(\overline{\theta} \in U\) and \(U\) is closed. In summary, \(U\) is compact. Take \(\hat{\theta} \in U\). Since \(G\left(\theta^k\right) \geq G(\hat{\theta})\) \(\forall k \ge 0\). We consider the following two cases: 
%

 If \(G(\hat{\theta}) = G(\theta^k)\) for some \(k \ge 0\), then for \(l > k\), we also have \(G(\hat{\theta}) = G(\theta^l)\), so \(\hat{\theta}\) is an accumulation point and it is a stationary point. 
    
    If \(G(\hat{\theta}) < G(\theta^k)\), using Proposition \ref{smooth_clarke} and Theorem \ref{semismooth}, we deduce \(-p_k = \displaystyle\sum_{i \in J_{\delta}(\theta^k)} u^i \nabla g_i\left(\theta^k\right) \in \partial^C G(\theta^k)\). After that, using Proposition \ref{prop41}, it implies that
\[
\left\langle -p_k, \hat{\theta} - \theta^k \right\rangle \leq 0 \quad \forall k \geq 0,
\]
then we have 
\[
\sum_{i \in J_{\delta}(\theta^k)} u^i \left\langle \nabla g_i\left(\theta^k\right), \hat{\theta} - \theta^k \right\rangle \leq 0 \quad \forall k \geq 0,
\]
\[
\begin{aligned}
\left\|\theta^{k+1} - \hat{\theta}\right\|^2 & \leq \left\|\theta^k - \hat{\theta}\right\|^2 - \left\|\theta^{k+1} - \theta^k\right\|^2 + 2 \alpha_k \sum_{i \in J_{\delta}(\theta^k)} u^i \left\langle \nabla g_i\left(\theta^k\right), \theta^k - \theta^{k+1} \right\rangle.
\end{aligned}
\]

Using Lemma \ref{sequence_lem}, when \(u_k = \left\|\theta^{k+1} - \hat{\theta}\right\|^2\), \(v_k = \displaystyle\sum_{i \in J_{\delta}(\theta^k)} u^i \left\langle \nabla g_i\left(\theta^k\right), \theta^k - \theta^{k+1} \right\rangle\), we imply that \(\left\{\left\|\theta^k - \hat{\theta}\right\|\right\}\) converges for all \(\hat{\theta} \in U\). Because \(\left\{\theta^k\right\}\) is bounded, there exists a subsequence \(\left\{\theta^{k_j}\right\} \subset \left\{\theta^k\right\}\) such that \(\displaystyle\lim_{i \rightarrow +\infty} \theta^{k_i} = \bar{\theta} \in U\). Thus,
\[
\displaystyle\lim_{k \rightarrow +\infty} \left\|\theta^k - \bar{\theta}\right\| = \lim_{i \rightarrow +\infty} \left\|\theta^{k_i} - \bar{\theta}\right\| = 0.
\]

Remark that each limit point of \(\left\{\theta^k\right\}\) is a stationary point of the problem. Hence, the sequence \(\left\{\theta^k\right\}\) converges to \(\bar{\theta}\) - a stationary point of the Problem \eqref{MP}. 

\end{proof}

\begin{theorem}\label{lemma33-6}
   In the pseudo-convex case, i.e. $g_i(\theta)$ are pseudoconvex, then $\theta^*$ is the global minimum point of the problem.
\end{theorem}









\begin{proof}
Consider $\delta$ is small enough and let \(J(\theta^*) = \{ i = \overline{1, m}, g_i(\theta^*) =  G(\theta^*)\}\), \(\overline{u}^i \ge 0, i = \overline{1, m}\), then for \(i \notin J(\theta^*)\), we have \(\overline{u}^i = 0\) and the following statements are obtained

\[\begin{cases}
    \sum_{i \in J(\theta^*)} \overline{u}^i = 1, \\[0.7cm]
    \sum_{i \in J(\theta^*)} \overline{u}^i \nabla g_i(\theta^*) = 0,\\[0.7cm]
    \overline{u}^i (g_i(\theta^*) - G(\theta^*)) = 0, \; i = \overline{1, m}.
\end{cases}\]

Using Proposition \ref{pseudo_semi}, for \(g_i\), \(i = \overline{1, m}\) are differentiable pseudoconvex functions, we obtain \(g_i\), \(i = \overline{1, m}\) are semiconvex functions.

Using Proposition \ref{smooth_clarke}, for \(g_i(\theta)\), \(i = \overline{1, m}\) are locally Lipschitz, then the Clarke differential of each \(g_i\) is \(\partial^C g_i(\theta^*) = \{ \nabla g_i(\theta^*)\}\).

For \(g_i\), \(i = \overline{1, m}\) are locally Lipschitz functions, using Theorem \ref{semismooth}, we have
\[
    \partial^C G(\theta^*) = \text{Conv} \{ \partial g_i(\theta^*), i \in J(\theta^*) \}.
\]

It is noticeable that \(0 = \sum_{i =1}^m \overline{u}^i \nabla g_i(\theta^*) \in \text{Conv} \{ \partial^C g_i(\theta^*), i \in J(\theta^*) \} = \partial^C G(\theta^*)\).

Using Theorem \ref{semismooth}, when \(G(\theta)\) is a semiconvex function, then it is quasidifferentiable at \(\theta\). Because \(G\) is quasidifferentiable at \(\theta\), based on the definition, \(G'(\theta, d)\) exists for all \(d \in \R^n\) and \(G'(\theta, d) = G^0(\theta, d)\).

Then, following the definition of \(\partial^C G(\theta^*)\), we have \(0 \in \partial^C G(\theta^*)\) which leads to \(G^0(\theta^*, d) \ge 0\) for all \(d \in \R^n\). Since \(G(\theta)\) is semiconvex at \(\theta\), \(G(\theta)\) is quasidifferentiable at \(\theta\) and \(0 \le G^0(\theta^*, d) = G'(\theta^*, d) = \displaystyle\lim_{t \to 0^+} \frac{G(\theta^* + td) - G(\theta^*)}{t}\) for all \(d \in \R^n\), then \(G(\theta^* + td) \ge G(\theta^*)\) for all \(d \in \R^n\) implies that \(\theta^*\) is a local minimum of the Problem \eqref{MP}.

It is noticeable that because \(\Omega = \{\theta \in \R^n : G(\theta) \le G(\theta_0)\}\), if the Problem (MP) has a global minimum, then the global solution is in \(\Omega\). Besides, consider all \(d \in \R^n\) such that \(\theta^* + d \in \Omega\), then due to \(G\) is semiconvex at \(\theta^*\), based on the definition, \(G'(\theta^*, d) \ge 0\), hence \(G(\theta^* + d) \ge G(\theta^*)\). Thus, for each \(\theta \in \Omega\), there exists \(d \in \R^n\) such that \(\theta = \theta^* + d\) and \(G(\theta) \ge G(\theta^*)\), \(\forall \theta \in \Omega\), and \(\theta^*\) is the global minimum of the Problem \eqref{MP}. \end{proof}   
\section{Multiobjective Optimization Problem}
\subsection{Notations}
We are examining the multicriteria optimization problem as described below:
\begin{equation}
\min _{\theta \in \mathcal{X}} g(\theta)=\left(g_1(\theta), g_2(\theta), \ldots g_m(\theta)\right)^{\top}, \tag{MOP} \label{MOP}
\end{equation}
in there, $\mathcal{X} \subseteq \mathbb{R}^n \neq \emptyset$ is called the feasible set, and $g_i: \mathbb{R}^n \rightarrow \mathbb{R}$, $i = \overline{1, m}, m \ge 2$ represent the $m$ objective functions. The set $\mathcal{Y} = \{g(\theta) \in \mathbb{R}^m : \theta \in \mathcal{X}\}$ represents the (feasible) objective set. The spaces $\mathbb{R}^n$ and $\mathbb{R}^m$ are described as the decision and objective spaces, respectively.

To compare the values of objective functions, we mention some relations with regard to cone order (Ehrgott, 2005).

\begin{definition}{\cite{dachert2017efficient}}
For two objective vectors $\xi^1$ and $\xi^2$ in the objective space $\mathcal{Y} \subseteq \mathbb{R}^m$, we define:
\begin{itemize}
\item $\xi^1 \leqq \xi^2$ (meaning $\xi^1$ weakly dominates $\xi^2$) if and only if $\xi_i^1 \leq \xi_i^2$ for every $i = 1, 2, \ldots, m$.
\item $\xi^1 \leq \xi^2$ (meaning $\xi^1$ dominates $\xi^2$) if and only if $\xi^1 \leqq \xi^2$ and $\xi^1 \neq \xi^2$.
\item $\xi^1 < \xi^2$ (meaning $\xi^1$ strictly dominates $\xi^2$) if and only if $\xi_i^1 < \xi_i^2$ for every $i = 1, 2, \ldots, m$.
\end{itemize}
\end{definition}

When neither $\xi^1 \geq \xi^2$ nor $\xi^2 \geq \xi^1$, $\xi^1$ and $\xi^2$ are considered incomparable.

We now introduce the concept of dominance relations for decision vectors.

\begin{definition}{\cite{audet2008multiobjective}}
For two decision vectors $\theta^1$ and $\theta^2$ in the feasible set $\mathcal{X} \subseteq \mathbb{R}^n$, we define:
\begin{itemize}
\item $\theta^1 \preceq \theta^2$ (meaning $\theta^1$ weakly dominates $\theta^2$) if and only if $g(\theta^1) \leqq g(\theta^2)$.
\item $\theta^1 \prec \theta^2$ (meaning $\theta^1$ dominates $\theta^2$) if and only if $g(\theta^1) \leq g(\theta^2)$ and $g(\theta^1) \neq g(\theta^2)$.
\item $\theta^1 | \theta^2$ (meaning $\theta^1$ and $\theta^2$ are incomparable) if neither $\theta^1$ weakly dominates $\theta^2$ nor $\theta^2$ weakly dominates $\theta^1$.
\end{itemize}
\end{definition}

With these relations, we define the concept of a solution in the multi-objective optimization framework.

\begin{definition}{\cite{ehrgott2005multicriteria}}
A vector $\theta \in \mathcal{X}$ is considered a Pareto-optimal solution if no other vector in $\mathcal{X}$ can dominate it. The collection of all Pareto-optimal solutions is referred to as the Pareto set, represented by $\mathcal{X}_P$, and the corresponding set of objective values is known as the Pareto front, denoted by $\mathcal{Y}_P$.
\end{definition}

\begin{definition}{\cite{zhao2021convergence}}
    A point $\theta^*$ is a Pareto stationary point (or a Pareto critical point) of $g$ if 
\[
J_g(\theta^*)(\R^n - \theta^*) \cap (-\R^{m}_{++}) = \emptyset,
\]

where the Jacobian matrix of $g$ at $\theta^*$, i.e $J_g(\theta^*)$ is given by $J_g(\theta) = (\nabla g_1(\theta), \ldots, \nabla g_m(\theta))^T.$ 
\end{definition}

\begin{proposition}{\cite{ehrgott2005multicriteria}}
\label{prop22}
    The optimal solution of the Problem \eqref{MP} is the weakly efficient point (weakly pareto solution) of the Problem \eqref{MOP}.
\end{proposition}

After that, we consider some well-known techniques to address the Problem \eqref{MOP}, as described below:
\subsection{Linear Scalarization} 
One of the most prevalent approaches that have been usually used for tackling the Problem \eqref{MOP} is linear scalarization. This mechanism employs a linear combination of all objective function components and transform into the minimization problem of this term:
$$
\min _\theta G(\theta)=\sum_{i=1}^l w_i g_i(\theta),
$$
where in the given context, $w_i$ denotes the weight associated with the $i$-th objective function. Although this approach is relatively straightforward, it does have certain limitations, particularly concerning the initialization of the weight vector.

Typically, in practice, weights are set prior to the execution of the method. Consequently, the overall effectiveness of the approach is significantly influenced by the choice of these weights. Determining the optimal weights is an unresolved issue and remains a challenging task, even for experts in Multi-Objective Optimization (MOP) who are well-acquainted with the specific problem being addressed.


\subsection{Tchebycheff Scalarization}
Tchebycheff scalarization is one of the approaches for solving the Problem \eqref{MOP}. This approach is particularly useful for multiobjective optimization as it transforms the problem into a scalar one, making it more manageable to solve. The core concept of Tchebycheff scalarization closely parallels the process of solving a minimax problem, where the goal is to minimize the maximum deviation from the ideal solution. This concept will be further elaborated upon and illustrated in the following section, where we delve into its application and effectiveness in solving complex multiobjective optimization problems.

The explicit form of the Problem \eqref{MOP_0} is described below:

\begin{alignat*}{3}\label{MOP_0}
\min & \; \tau, \tag{$MOP_0$} \\
\text { s.t } & g(\theta)- v - \tau \hat{d} \leq 0; \\
& \theta \in \R^n, \; \tau \in \R, 
\end{alignat*}
where $\hat{d}, v \in \R^m$ be the initialized scalar and translational vectors, respectively.

The Problem \eqref{MOP_0} is equivalent to the following problem
\begin{alignat*}{2}\label{MOP_1}
\min_{\theta} & \max_{l}  \left\{\frac{g_l(\theta)-v_l}{\hat{d}_l} \mid l=1, \ldots, m\right\}, \tag{$MOP_1$} \\
\text { s.t } & \theta \in \R^n.
\end{alignat*}

By initializing $v$ and $d$, the Problem \eqref{MOP_1} may be deduced to solving approximately the efficient set. In the next section, we are going to discuss the case when $v = (0, \ldots, 0)^T$ and $\hat{d} = (1, \ldots, 1)^T$. Then, an efficient point of the multiobjective optimization can be solved by the Problem \eqref{MP}.

\subsection{Non-monotone adaptive algorithm}
Next, we propose the following algorithm to tackle the multiobjective optimization based on the results obtained from the Algorithm \ref{algo1}:

\begin{algorithm}[H]
\caption{Adaptive Non-monotone strategies for solving the \eqref{MOP} problem}
\hspace*{\algorithmicindent} \textbf{Input:} Approximation point $\theta^0$, translation vector $v$.\\
\begin{algorithmic}[1]
\State Using reference-based techniques to find the set of reference vectors $u_i$, $i = \overline{1, K}$. 
\FOR{$i=1$ to $K$}
\State For each reference vector, solving the Problem \eqref{MP_MOP} based on the Algorithm \ref{algo1}:
\State \begin{equation*}
\label{MP_MOP}\tag{MMOP}
    \min_{\theta_{i}} \max_j \; \left\{\dfrac{g_j(\theta_i) - v_j}{d_j}, j = 1, \ldots, m\right\}
\end{equation*}
in there $d_j$ be the $j-$index regarding to the reference vector $u_i$. 
\ENDFOR
\end{algorithmic}    
\label{algo2}
\hspace*{\algorithmicindent} \textbf{Output:} The solution points of the Problem \eqref{MOP}. 
\end{algorithm}

In this paper, we use the Das-Dennis Algorithm to find a set of equally reference vectors $u_i, i = \overline{1, K}$ so that we can strongly show the controllable capabilities vectors onto the Pareto front of our method. In addition, we choose translation vector $v$ as zero vector with dimension equal to the number of component functions of the multiobjective optimization.  

Next, we will discuss the convergence of Algorithm \ref{algo2} under non-convex, quasiconvex, and pseudoconvex conditions in relation to Problem \eqref{MOP}.

\subsection{Convergence of the algorithms}
\begin{theorem} 
For the sequence $\{\theta^k\}$ produced by the Algorithm \ref{algo1}, we establish the three following statements:
\begin{enumerate}
	\item If $g_i$, $i = \overline{1, m}$ is continuous differentiable and nonconvex, then the accumulation point (if exists) $\theta^*$ is the stationary point of the problem \eqref{MP} and the Pareto critical point of the Problem \eqref{MOP}.
	\item If $g_i$, $i = \overline{1, m}$ is continuous differentiable and quasiconvex, then the accumulation point $\theta^*$ is the stationary point of the Problem \eqref{MP} and the Pareto critical point of the Problem \eqref{MOP}.
	\item If $g_i$, $i = \overline{1, m}$ is continuous differentiable and pseudoconvex, then $\theta^*$ is the global minimum of the Problem \eqref{MP} and the weakly efficient point of the Problem \eqref{MOP}.
\end{enumerate}
\end{theorem}

\begin{proof} 
Let consider the following two cases:
\begin{enumerate}
    \item \textit{\textbf{Nonconvex and Quasiconvex case}} 
\end{enumerate}

Based on the definition of the stationary point of the (MOP) problem, we have to prove that 
\begin{equation*}
J_g(\theta^*)(\R^n - \theta^*) \cap (-\R^{m}_{++}) = \emptyset,
\end{equation*}
where \(J_g(\theta^*)\) is the Jacobian matrix of \(G\) at \(\theta^*\) given by \(J_G(\theta) = (\nabla g_1(\theta), \ldots, \nabla g_m(\theta))^T\).

Let us consider \(\theta \in \R^n\), suppose that this statement is contrary. Then we have 
\(\left\langle \nabla g_i(\theta^*), (\theta - \theta^*) \right\rangle < 0, \forall i = \overline{1, m}\). 
From the previous Theorem, for the accumulation point generated by algorithm \ref{algo1}, we have the following condition 
\begin{equation*}
\sum_{i =1}^m \overline{u}^i \nabla g_i(\theta^*) = 0.
\end{equation*}

Hence, we have \(0 = \left\langle \sum_{i =1}^m \overline{u}^i \nabla g_i(\theta^*) , \theta - \theta^* \right\rangle = \sum_{i =1}^m \overline{u}^i \left\langle \nabla g_i(\theta^*), \theta - \theta^* \right\rangle < 0\), which is contrary, so there exists \(i = \overline{1, m}\) such that \(\langle \nabla g_i(\theta^*), \theta - \theta^* \rangle \ge 0\). Thus, \(\theta^*\) is the stationary point of Problem \eqref{MOP}. 

\begin{enumerate}
    \item[(2)] \textit{\textbf{Pseudoconvex case}} 
\end{enumerate}

From Lemma \ref{lemma33-6}, we obtain that \(\theta^*\) is the global minimum of the \eqref{MP} problem. In addition, By using Proposition \ref{prop22}, we conclude that \(\theta^*\) is the weakly efficient point of the Problem \eqref{MOP}.

\end{proof}


\section{Numerical Results}
In this section, we assess the effectiveness of the proposed Algorithm \ref{algo2}, which builds upon the non-monotone adaptive strategies of Algorithm \ref{algo1}, by applying it to several medium- to small-scale, unconstrained multiobjective optimization problems. The numerical experiments were executed on NVIDIA V100 GPUs with 80 GB of memory, using Python 3.10 within a Linux operating environment. For these experiments, we configured the hyperparameters with $\varepsilon = 0.4$, $\sigma = 0.9$, $\eta_k = 0.01^k$, and an initial $\alpha_0 = 1$ for the first iteration. The translation vector $v$ was set to a zero vector, and the scalar vectors $u_i$, for $i = \overline{1, K}$, were distributed equally using the Das-Dennis Algorithm, where $K$ represents the number of reference vectors. The test problems are divided into three scenarios: the first example evaluates the algorithm's performance on a convex \eqref{MOP} problem; the second tests its effectiveness on a non-convex \eqref{MOP} problem in a higher-dimensional space (20 dimensions); and the third explores a case involving both a convex and a pseudoconvex component. Each scenario employs consistent techniques regarding the reference-based vectors. Detailed explanations of these experiments are provided below:

\begin{example}
Let consider the convex multiobjective optimization

    $$\begin{aligned}
\label{4.1}
& \operatorname{min} g(\theta)=\left\{\frac{1}{25} \theta_1^2+\frac{1}{100}\left(\theta_2-\frac{9}{2}\right)^2, \frac{1}{25} \theta_2^2+\frac{1}{100}\left(\theta_1-\frac{9}{2}\right)^2\right\}, \\
& \text { s.t. } \theta \in \mathbb{R}^2.
\end{aligned}$$
\end{example}

The goal is to minimize the two objective functions simultaneously while adhering to the constraints. The results obtained through our algorithm, which are aligned with the provided reference vectors, are shown in Figure \ref{fig:4-1}. 

\begin{figure}[H]
    \centering
    \renewcommand{\figurename}{Figure}
    \includegraphics[width = 0.81\textwidth]{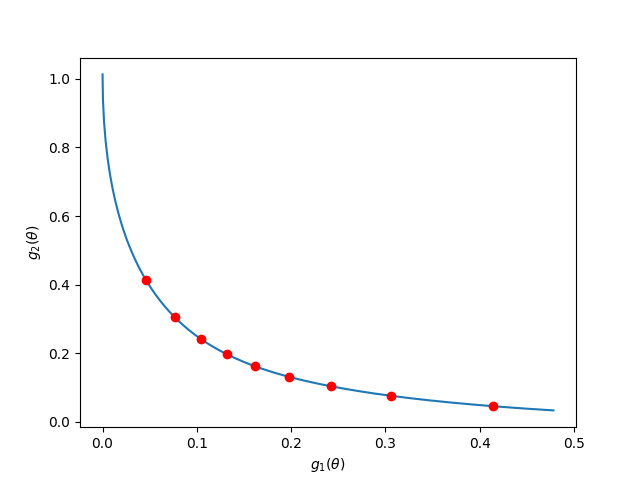}
    \caption{Results from Algorithm \ref{algo2} applied to Example \hyperref[4.1]{5.1}}
    \label{fig:4-1}
\end{figure}

This example serves as a clear demonstration of the algorithm’s effectiveness in solving convex multiobjective problems by accurately identifying the optimal solutions corresponding to the given reference vectors.



\begin{example}{(Lin et al. \cite{lin2019pareto})}
   Let analyse the following nonconvex, unconstrained multicriteria problem
    $$
\begin{aligned}
\label{4.2}
& \operatorname{min} g(\theta)=\left\{1-\exp({-\sum_{j=1}^n\left(\theta_j-\frac{1}{n}\right)^2}), 1-\exp({-\sum_{j=1}^n\left(\theta_j+\frac{1}{n}\right)^2})\right\}, \\
& \text { s.t. } x \in \mathbb{R}^{n}.
\end{aligned}
$$
\end{example}

In this example, we investigate the performance of our proposed algorithm on a complex nonconvex, unconstrained multicriteria optimization problem originally presented by Lin et al. \cite{lin2019pareto}. The objective functions are designed to challenge optimization techniques due to their nonconvex nature. 

\begin{figure}[H]
    \centering
    \renewcommand{\figurename}{Figure}
    \includegraphics[width = 0.81\textwidth]{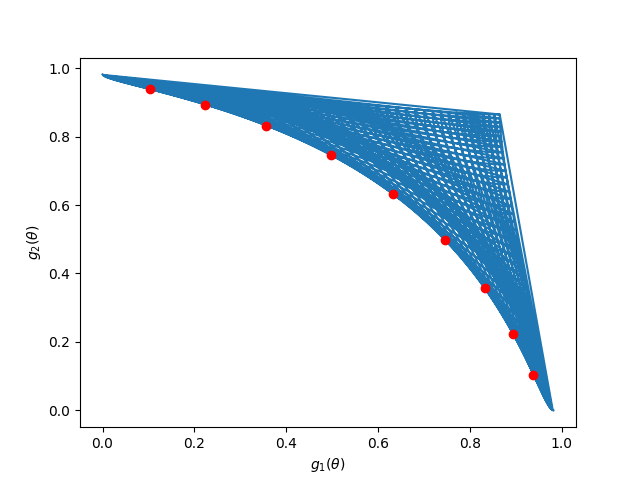}
     \caption{Results from Algorithm \ref{algo2} applied to Example \hyperref[4.2]{5.2}}
    \label{fig:4-2}
\end{figure}

By setting the dimensionality $n$ to $20$ and using nine reference points, we demonstrate the algorithm's ability to effectively handle and solve high-dimensional, nonconvex problems. The results, as shown in Figure \ref{fig:4-2}, provide a clear indication of the algorithm's robustness and accuracy in this context.

\begin{example}{(Zhao et al. \cite{zhao2021convergence})}
    Consider the multiobjective optimization where the objective functions are defined as follows:
    $$
\begin{aligned}
\label{4.4}
& \text{min } g(\theta) = \left\{\frac{1}{\theta_1^2+\theta_2^2+1}, \theta_1^2+3 \theta_2^2+1 \right\}\\
& \text { s.t. } \theta \in \mathbb{R}^{2}.
\end{aligned}
$$
\end{example}

In this example, we aim to demonstrate the accuracy of Algorithm \ref{algo2} in addressing non-convex multiobjective problems that involve both convex and pseudoconvex components. The numerical experiments validate the effectiveness of our theoretical results related to Algorithms \ref{algo1} and \ref{algo2}. 

\begin{figure}[H]
    \centering
    \renewcommand{\figurename}{Figure}
    \includegraphics[width = 0.81\textwidth]{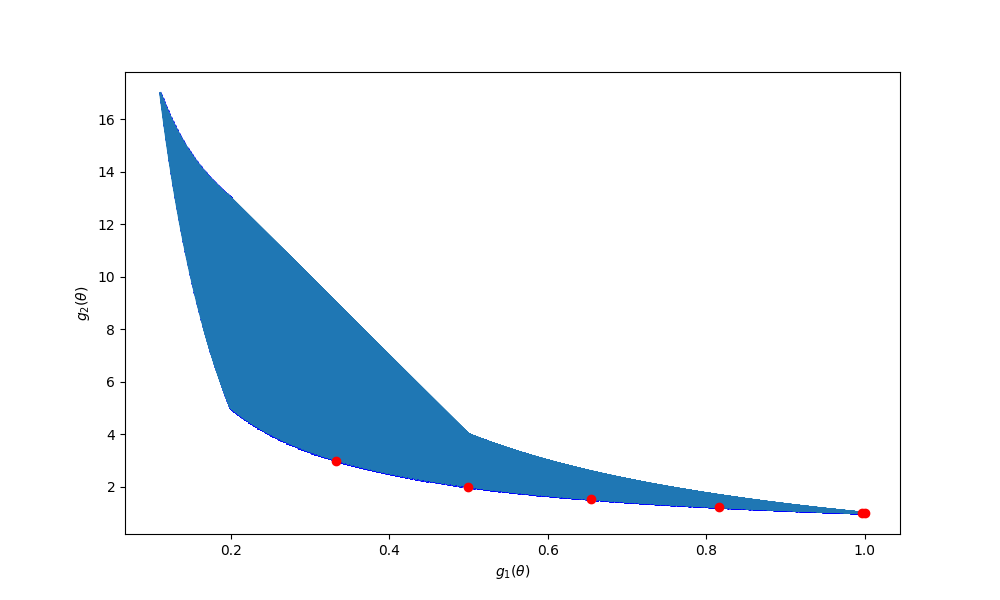}
     \caption{Results from Algorithm \ref{algo2} applied to Example \hyperref[4.4]{5.3}}
    \label{fig:4-3-1}
\end{figure}

These results are illustrated in Figure \ref{fig:4-3-1}, which displays nine reference vectors of varying scales corresponding to the objective component function values.

\section{Conclusion}

This paper investigates the steepest descent method combined with a non-monotone adaptive strategy for tackling non-convex multiobjective optimization problems. With relatively mild assumptions, this approach - featuring the steepest descent technique and a non-monotone adaptive step size - is shown to converge to solutions that are efficient (Pareto stationary) and meet certain technical criteria (quasiconvex or pseudoconvex). Furthermore, the method's effectiveness is demonstrated through numerical experiments addressing Problem \eqref{MOP} based on Problem \eqref{MP}, highlighting its accuracy and efficiency across various types of multiobjective optimization problems and different dimensionalities.

This work lays the foundation for future practical applications of this non-monotone method. Additionally, by employing various reference-based techniques instead of the Das-Dennis algorithm, a broad spectrum of weakly efficient points on the Pareto front can be obtained, irrespective of the scale of the objective component functions. Moreover, adapting this method to handle multiple nonsmooth objectives presents a promising avenue for future research.

\bibliographystyle{tfnlm}
\bibliography{refs}

\end{document}